\documentclass[11pt]{article}
\usepackage{amsrefs, amsmath, amsfonts, amsthm, latexsym, hyperref, comment}
\usepackage[margin=1in]{geometry}
\usepackage{enumerate}
\usepackage{authblk}
\usepackage[dvipsnames]{xcolor}
\usepackage{float}

\usepackage{pgfplots}
\usetikzlibrary{decorations.pathreplacing}
\usetikzlibrary{intersections}

\usepackage[margin=2cm]{caption}
\usepackage{verbatim}

\newtheorem{thm}{Theorem}
\newtheorem{innercustomthm}{Theorem}
\newenvironment{customthm}[1]
  {\renewcommand\theinnercustomthm{#1}\innercustomthm}
  {\endinnercustomthm}

\newtheorem{clm}{Claim}[section]
\newtheorem{obs}[clm]{Observation}
\newtheorem{lem}[clm]{Lemma}
\newtheorem{prop}[clm]{Proposition}

\newtheorem{cor}[clm]{Corollary}
\newtheorem{conj}{Conjecture}

\newtheorem{rmk}{Remark}

\newcommand{\N}{\mathbb{N}}

\newcommand{\R}{\mathbb{R}}
\newcommand{\Pro}{\mathbb{P}}

\newcommand{\ep}{\varepsilon}
\newcommand{\al}{\alpha}

\newcommand{\lm}{\lambda}
\newcommand{\two}{(\lm^*_{n-1},\lm^*_n)}

\newcommand{\PLH}{{\mkern-2mu\times\mkern-2mu}}

\makeatletter
\newcommand{\subjclass}[2][1991]{%
  \let\@oldtitle\@title%
  \gdef\@title{\@oldtitle\footnotetext{#1 \emph{Mathematics subject classification.} #2}}%
}
\newcommand{\keywords}[1]{%
  \let\@@oldtitle\@title%
  \gdef\@title{\@@oldtitle\footnotetext{\emph{Key words and phrases.} #1.}}%
}
\makeatother

\title{Mean and Minimum of Independent Random Variables}

\author{Naomi Dvora Feldheim\thanks{Bar-Ilan University, Ramat-Gan, Israel. email: naomi.feldheim@biu.ac.il. Research supported in part by the Institute of Mathematics and Its Applications funded by the NSF, and by an NSF postdoctoral fellowship at Stanford University.}
\hspace{2pt}  and Ohad Noy Feldheim\thanks{Hebrew University, Jerusalem, Israel. email: ohad.feldheim@mail.huji.ac.il. Research supported in part by the Institute of Mathematics and Its Applications funded by the NSF and by a postdoctoral fellowship at Stanford University.}
}

\date{}

\begin{document}
\maketitle
\begin{abstract}
We show that any pair  $X, Y$ of independent non-compactly supported random variables on $[0,\infty)$ satisfies 
$$\liminf_{m\to\infty}\Pro(\min(X,Y) >m \,| \,X+Y> 2m) =0.$$
We conjecture multi-variate and weighted generalizations of this result, and prove them under the additional assumption that the random variables are identically distributed.
\end{abstract}

\section{Introduction}

By the simple inequality $\min(X,Y)\le \frac{X+Y}{2}\le \max(X,Y),$ it is evident that for any
pair of non-negative independent random variables $X, Y$, for all $m\ge 0$ we have
$$\Pro\Big(\min(X,Y)>m\Big)\le \Pro\left(\frac{X+Y}{2}>m\right)\le \Pro\Big(\max(X,Y)>m \Big). $$
Consider the asymptotic behavior of these inequalities when $m\to\infty.$ It is not hard to construct examples for which $\Pro\big(\frac{X+Y}{2}>m\big) \asymp \Pro\big(\max(X,Y)>m \big)$ (here $A_m \asymp B_m$ indicates that for all $m>0$ we have $c<A_m/B_m < C$ for some $0<c<C<\infty$).
For example if $X$ and $Y$ are identically distributed with
$\Pro\big(X>m\big)=1/\log (m+e)$ then
$$\Pro\Big(\max(X,Y)>m \Big) \le 2\ \Pro\Big(X>m \Big) \le 10\ \Pro\Big(X>2m \Big)\le 10\ \Pro\left(\frac{X+Y}{2}>m\right) . $$

It is therefore natural to ask whether it is ever the case that $\Pro\big(\min(X,Y)>m\big)\asymp \Pro\big(\frac{X+Y}{2}>m\big)$. Our main result, confirming a conjecture of Alon \cite{Aconj}, is that this is {\bf never} possible.

\begin{thm}\label{thm: main}
Let $X,Y$ be independent random variables on $\R_+$, which are not compactly supported. Then:
\begin{equation}\label{eq: main}
\limsup_{m\to \infty} \frac{\Pro\left(\frac{X+Y}{2}> m\right)}{\Pro(X>m)\Pro(Y>m)} =\infty.
\end{equation}
\end{thm}
In other words, any independent, unbounded, non-negative random variables $X, Y$ satisfy:
\[
\liminf_{m\to\infty} \Pro\left( \min(X,Y)>m \: \Big|\: \frac{X+Y}{2} >m \right)=0.
\]

We remark that the $\limsup$ in the theorem is necessary: there may be an unbounded set of numbers $m$ such that the ratio $\Pro(X+Y>2m)/\Pro(\min(X,Y)>m)$ gets arbitrarily close to $1$.
However, as will become evident from the proof, when the tail distribution is either log convex or log concave the limit is guaranteed to exist. 

Theorem~\ref{thm: main} is limited to two variables and to unweighted averages. It is natural to ask if a similar statement could hold for an arbitrarily weighted average of several variables. In Section~\ref{sec: high} we conjecture such a generalization, which we later prove for the case when the variables are identically distributed. Our results could also be viewed as anti-concentration statements for product measures, a point of view which calls for additional, perhaps more bold conjectures. This is further discussed in Section~\ref{sec: disc} where we also relate our work to the ``123 comparison inequality'' of Alon and Yuster~\cite{AY} and its generalizations.

An application of our result appears in a follow-up paper \cite{FFsocial}, and briefly discussed in Section~\ref{sec: appl}.
It concerns with convergence of a model for evolving social groups introduced by Alon et al. in \cite{social}. 

Finally in Section~\ref{sec: idea} we provide an overview of our methods along with an outline for the rest of the paper.

\subsection{High dimensions and weighted averages}\label{sec: high}
The following is a natural generalization of Theorem \ref{thm: main}.
\begin{thm}\label{thm: general}
Let $X_1, \dots, X_n$ be i.i.d. random variables with a non-compactly supported distribution on $\R_+$.
For any $(\lm_1,\dots,\lm_n)\in(0,1)^n$ with $\sum_j \lm_j=1$ we have
\begin{equation}\label{eq: gen}
\frac {\Pro\left(\sum_{j=1}^n \lm_j X_j > m \right) }{\Pro\left(X_1> m \right)^n }\ge \al_n(m)\quad \text{\emph{ where }\quad$\limsup_{m\to\infty} \al_n(m)=\infty$}
\end{equation}
\end{thm}

It remains open to show that Theorem~\ref{thm: general} is true when $X_1,\dots,X_n$ are merely independent (not necessarily identically distributed). While we believe this to be true, our proofs do not extend to this case. Indeed, to prove Theorem~\ref{thm: main} in the non-i.i.d. case we employ a symmetry that exists only in the case of two equal weights.
This generalization would, however, follow from the following, which is our main conjecture.

\begin{conj}\label{conj: 2 weights}
Let $X, Y$ be independent random variables on $\R_+$ which are not compactly supported, and let $\lm \in (0,1)$. Then:
\begin{equation*}
\limsup_{m\to \infty} \frac{\Pro(\lm X+(1-\lm) Y> m)}{\Pro(X>m)\Pro(Y>m)} =\infty.
\end{equation*}
\end{conj}

In fact, Conjecture~\ref{conj: 2 weights} would yield a much more general result concerning product measures and arbitrary norms, stated as follows.
\begin{conj}
Let $n\in \N$, and $\|\cdot\|_K$ be any norm in $\R^n$. Let $X_1, \dots, X_n$ be independent, non-compactly supported random variables on $\R_+$. Then for any vector $(a_1,\dots,a_n)\in (0,\infty)^n$ we have:
\begin{equation*}
\limsup_{m\to\infty} \frac{\Pro( \| (X_1,\dots,X_n) \|_K > m \| (a_1,\dots,a_n) \|_K ) } {\prod_{j=1}^n \Pro( X_j > m\, a_j ) } = \infty.
\end{equation*}
\end{conj}

\subsection{Distribution-free comparison inequalities and anti-concentration}\label{sec: disc}

It is instructive to view our results in light of \emph{distribution-free comparison inequalities} which were obtained for other events.
The interest in such probabilistic inequalities and their relation with combinatorics goes back at least to the 1980's (see the survey by Katona \cite{K} from that time).
Perhaps the most celebrated comparison inequality is ``the 123 theorem'' by Alon and Yuster~\cite{AY}, which states that for any i.i.d. random variables $X,Y$ we have
\[
\Pro\left( |X-Y| \le 2 \right) < 3\, \Pro( |X-Y| \le 1).
\]
The authors extended this result to compare the events $\{|X-Y|\le b\}$ and $\{|X-Y|\le a\}$ for any $a,b>0$, with a universal optimal constant.
To see the connection with our result more clearly, we rewrite Theorem~\ref{thm: main} as follows:
\begin{customthm}{1*}\label{thm: comp}
For any independent random variables $X, Y$ with non-compactly supported distribution $\mu$ on $\R_+$, there is no number $c$ such that for all $m>0$,
\begin{equation*}
\Pro\left( \frac {X+Y}2>m\right)< c\, \Pro(\min(X,Y)>m)
\end{equation*}
\end{customthm}
Thus, there is no comparison inequality between the tail-distribution function of the average $\frac 1 2 (X+Y)$ and that of the minimum $\min(X,Y)$, even for a single fixed distribution (let alone with a universal constant).

It is interesting to note that the Alon-Yuster inequality was generalized and applied in other settings.
A work by Dong, Li and Li~\cite{DLL} gives a universal comparison inequality for sums and differences of i.i.d. random variables taking values in a separable Banach space. These inequalities were further generalized by Li and Madiman \cite{LM}, who also explored the connections with extremal combinatorial problems.
It is also worth mentioning an earlier work by Schulze and Weizs\"acher~\cite{SW}, which established one of those inequalities for $\R$-valued random variables, and applied it to derive the rate of decay of the crossing level probability of an arbitrary random walk with independent increments.

As pointed out in~\cite{LM}, general concentration phenomena may stem out of distribution-free inequalities. In our case,  Theorem~\ref{thm: main} may be viewed as an ``anti-concentration'' result for product measures. Roughly speaking, it states that any product measure on $\R_+^2$ cannot be too concentrated around the diagonal $\{ (x,x): x>0\}$.
In light of this discussion, it is natural to wonder if our anti-concentration bound has counterparts in other spaces.

\subsection{An application to evolving social groups} \label{sec: appl}
In a recent study by Alon~\cite{social}, the following family of models for exclusive social groups (referred to here as \emph{clubs}) was introduced. Let $r\in(0,1)$ and let $\mu$ be an arbitrary distribution on $[0,\infty)$ representing opinions in a population (say, political inclination between left and right).
In the \emph{$r$-quantile admission process with veto power},
the club starts with a single ``extreme left'' founding member with opinion $0$.
At every step two independent candidates, whose opinions are $\mu$-distributed, apply for admission.
Each member then votes for the candidate whose opinion is closer to his (breaking ties to the left).
If at least an $r$-fraction of the current club members prefer the left-most candidate then he is admitted, and otherwise
none of the candidates are admitted.

 In~\cite{social} the authors consider this model for $\mu$ which is uniform on $[0,1]$. They show that, somewhat surprisingly,
the model exhibits a phase transition at $r=1/2$. In particular, when $r<1/2$ the distribution of opinions converges almost surely to some fixed continuous distribution. At the same time, for $r>1/2$ as the club grows, only candidates closer and closer to $1$ are accepted and the club becomes ``extreme-right''.

It is natural to ask: ``How does this behavior depend on $\mu$, the distribution of the applicants' opinions? Does it matter if this distribution is compactly supported? Could it ever be that the $r$-quantile of the empirical distribution will drift towards infinity?''

The problem is intimately related to the one discussed here, since the probability that the next admitted member's opinion
will be further to the right than the current $r$-quantile is exactly
$$\Pro\left(\min(X,Y) > q_t\, \big| \, \frac{X+Y}2\ge q_t\right),$$
where $q_t$ is the $r$-quantile after $t$ candidates were admitted,
and $X$ and $Y$ are independent $\mu$ distributed random variables.
Theorem~\ref{thm: main} thus implies that as $q_t$ grows it has a strong drift towards the left.
With some additional work one can show that $q_t$ is almost surely bounded, for any distribution of opinions $\mu$. 
These steps are used in our paper \cite{FFsocial} to show that, in fact, the $r$-quantile almost surely converges, and hence the empirical distribution of the club converges to a (possibly random) limit distribution.

\subsection{Main ideas and outline}\label{sec: idea}
The protagonist of the proof of Theorem \ref{thm: main} is the log-tail function: $g(m)=-\log \Pro(X\ge m)$, which may be any non-decreasing function on $[0,\infty)$, such that $g(0)=0$ and $g(\infty)=\infty$.
The proof is founded on the case in which $X$ and $Y$ are identically distributed and $g$ is convex.
In this case we assume towards a contradiction that the ratio in \eqref{eq: main} is bounded.
We then show (in Lemma~\ref{lem: conv}) that this
implies a difference equation on $g^{-1}$ which forces it to increas to infinity on a finite interval, in contradiction with the assumption that $X$ is not compactly supported.

Next, towards obtaining the general theorem, we consider the case of $X$ and $Y$ which are identically distributed but $g$ is not necessarily convex.
We compare between the given measure and its ``nearest'' log-concave measure. This comparison classifies all $g$-s into three types: nearly convex, nearly concave, and oscillating.
More precisely, for general $g$, we define $h$ to be the \emph{convex minorant} of $g$ (i.e., the maximal non-decreasing convex function which is pointwise
less-equal to $g$). Our goal then is to draw properties from the relation between $h$ and $g$, in order to choose the points $m$ at which we claim the ratio in
\eqref{eq: main} to be big. Specifically, we divide the proof into three cases:
\begin{itemize}
\item (``nearly convex'') $\sup_{\R_+} (g-h)<\infty$: $g$ is in bounded distance from a convex function and the proof for convex $g$ may be applied.
\item (``nearly concave'') $\lim_{x\to\infty} (g-h)(x)=\infty$: Roughly speaking, in this case $g$ has a concave, sublinear behavior, which enables us to show that even
$ \frac{\Pro(X>2m)}{\Pro(X>m)^2}$ is asymptotically unbounded.
\item (``oscillating'') $\limsup_{x\to \infty} (g-h)(x)=\infty$ and $\liminf_{x\to\infty}(g-h)(x)<\infty$: here, we use the oscillations between $g$ and its convex minorant in order to find points for which the ratio in~\eqref{eq: main} is large.
\end{itemize}

The proof of Theorem~\ref{thm: main} in the non-i.i.d. case is based on a symmetrization argument, which reduces it to an i.i.d. case.
A similar scheme is used for Theorem~\ref{thm: general}, with appropriate generalizations to high dimensions and arbitrary weights.

The rest of the paper is organized as follows.
In Section~\ref{sec: main} we prove Theorem~\ref{thm: main} for i.i.d. random variables, while in Section~\ref{sec: main non iid} we extend it to any independent random variables.
Theorem~\ref{thm: general} concerning weighted averages of several i.i.d. variables is proved in Section~\ref{sec: general}.

\subsection{Acknowledgements}
We thank Noga Alon for introducing the problem and for useful discussions.
We are grateful to Mokshay Madiman and Jiange Li for pointing out the relation with comparison inequalities, and for suggesting generalizations which led to Theorem~\ref{thm: general}.
We are also grateful to Adi Gl\"ucksam, for suggesting the investigation of the non-i.i.d. case of Theorem~\ref{thm: main} and for helpful comments. Finally we thank the anonymous referee for useful comments which improved the presentation of the paper.


\section{Proof of Theorem~\ref{thm: main}: i.i.d. case}\label{sec: main}

This section is dedicated to the proof of Theorem~\ref{thm: main} under the additional assumption that $X,Y$
are identically distributed. In Section~\ref{sec: main: prel} we provide some preliminary tools. In Section~\ref{sec: main: conv}
we handle the nearly convex case, in Section~\ref{sec: main: conc} -- the nearly concave case and in Section~\ref{sec: main: osci} -- the remaining oscillating case.
Since these cases are exhaustive, the theorem follows. The statements of this section will be used in Section~\ref{sec: main non iid} to prove the theorem in full generality.

\subsection{Preliminaries}\label{sec: main: prel}

{\bf Basic notation.}
Throughout Section \ref{sec: main}, we fix a non-compactly supported measure $\mu$ on $\R_+$, and let $X$ and $Y$ be two independent random variables with law $\mu$.
Define
\begin{equation*}
F(x):=\mu ((x,\infty))\quad\text{ and }\quad g(x):= -\log F(x).
\end{equation*}
Notice that $F:\R_+ \to (0,1]$ is right-continuous and non-increasing (with $F(0)=1$ and $F(\infty)=0$) and that $g:[0,\infty)\to [0,\infty)$ is right-continuous and non-decreasing (with $g(0)=0$ and $g(\infty)=\infty$).

\medskip
{\bf Lebesgue-Stieltjes measure.} Since $g$ is non-decreasing, it defines a Borel measure $|\cdot|_g$  (called the \emph{$g$-Lebesgue-Stieltjes measure} or just the \emph{$g$-measure}).
This measure is determined by its operation on intervals, that is: $|[\al,\beta]|_g=g(\beta)-g(\al-)$ for any $0\le \al\le \beta$.
{It is possible to estimate $|A|_g$ for any measurable set $A$ using the following lemma}.

\begin{lem}\label{lem: hazava}
{
For any $a<b$,
\[
\log F(a)-\log F(b)\le -\int_a^b \frac{dF}{F}.
\]
}
\end{lem}
\begin{proof}[Proof of Lemma \ref{lem: hazava}]
Observe that for any $a\le b$ we have,
\[
-\int_{F(a)}^{F(b)}dt =
F(a)-F(b)\le 
F(a-) - F(b) = -\int_{a}^{b} dF .
\]
Using linearity of the integral we obtain that for any positive function $\varphi(x)$ we have
$$-\int_{F(a)}^{F(b)}\varphi(t) dt \le -\int_{a}^{b}\varphi(F(s)) dF(s) .$$
Plugging in $\varphi(t)=\frac{1}{t}$ we obtain
\[
\log F(a)-\log F(b)=\int_{F(b)}^{F(a)}\frac {dt}{t} \le -\int_a^b \frac{dF}{F},
\]
as required.
\end{proof}
From Lemma \ref{lem: hazava} we deduce that for any measurable set $A$ the following holds:
\begin{align}\label{eq: g(L)}
|A|_g= \int_{A} g'=\int_{A} (-\log F)'{\le} - \int_{A}\frac{dF}{F}=\int_{A} e^g d\mu.
\end{align}

\medskip
{\bf The set of $m$-symmetric $d$-concavity points.}
For a function $f:\R_+\to\R$ and parameters $m, d \ge 0$, define $L^f_{m,d}$, \emph{the set of $m$-symmetric $d$-concavity points of $f$}, as
\begin{equation*}
L^f_{m,d}= \left\{ \ell \in [0,2m]: \:  f(\ell)+f(2m-\ell) \le 2 \big(f(m) +d\big)   \right\}.
\end{equation*}
Observe that, by definition, $L^f_{m,d}$ is symmetric around $m$. A visual depiction of the definition of $L^f_{m,d}$ is provided in Figure~\ref{fig: L def}.

\begin{figure}[H]
\centering
\begin{tikzpicture}[scale=2, 
declare function={ f(\x) = 0.1+0.22*\x*\x;}]
     \draw[->] (-0.25, 0) -- (2.5,0) node  (xaxis) [right] {};
	\coordinate (yaxis) at (0,1);
     \draw[domain=0:1.75, smooth, variable=\x, blue,thick] plot ({\x},  {f(\x)});
     
     \draw [blue, smooth, thick] 
     (1.75, {f(1.75)}) to [in=210, out=45] (2.25,{f(2.25)}) to [in=210, out=30](2.5, {f(2.5)-0.1}) node[below] {$f$}; 
     
     \draw [blue, smooth, thick] 
     (-0.25,0.05) to [in=180, out=20] (0,{f(0)});

     \draw[thick] (0.5, {f(0.5)} ) coordinate (a) -- (1.75, {f(1.75)} ) coordinate (b);
      \coordinate (c) at (1.125, {f(1.125)} );
      \coordinate (ctag) at (1.125, {f(1.125)+0.19});
      
       \fill[blue] (c) circle (0.8pt);
       \fill[blue] (ctag) circle (0.8pt);
       \fill[black] (a) circle (0.8pt);
       \fill[black] (b) circle (0.8pt);

  \draw[dashed,blue] (yaxis |- c) node[left,font=\scriptsize] {$f(m)$}
        -| (xaxis -| c) node[below, font=\scriptsize] {$m$};
  \draw[dashed,blue] (yaxis |- ctag) node[left,font=\scriptsize] {$f(m)+d$}
        -| (xaxis -| ctag) node[below] {};
        
        \draw[dashed]  (a |- a) node[left] {} -| (xaxis -| a) node[below, font=\scriptsize] {$\ell$};
	 \draw[dashed]   (b |- b) node[left] {} -| (xaxis -| b) node[below, font=\scriptsize] {$2m-\ell$};
	 
	 \draw[thick] (0, {f(0)} ) coordinate (atag) -- (2.25, {f(2.25)} ) coordinate (btag);
       \fill[black] (atag) circle (0.8pt);
       \fill[black] (btag) circle (0.8pt);
       
        \draw[dashed]  (atag |- atag) -| (xaxis -| atag) node[below, font=\scriptsize] {$\ell'$};
	 \draw[dashed]  (btag |- btag) -| (xaxis -| btag) node[below, xshift=0.5em, font=\scriptsize] {$2m-\ell'$};
\end{tikzpicture}

\caption{
The defining criterion of $L^f_{m,d}$, the set of $m$-symmetric $d$-concavity points of a function~$f$. The value $\ell$ is in  $L^f_{m,d}$, as the average of $f(\ell)$ and $f(2m-\ell)$ is less than $f(m)+d$. On the other hand $\ell'$ is not in $L^f_{m,d}$ as it does not meet this criterion.
} \label{fig: L def}
\end{figure}

\bigskip
{\bf A useful reduction.}
We can now reduce Theorem~\ref{thm: main} to the following statement on $|L^g_{m,d}|_g$.
This will be our main tool for showing Theorem~\ref{thm: main} when $g$ is either nearly convex or oscillating.

\begin{lem}\label{lem: reduced}
If there exists $d\ge 0$ such that $\limsup_{m\to\infty} |L^g_{m,d}|_g = \infty,$   then $X$ and $Y$ satisfy \eqref{eq: main}.
\end{lem}
\begin{proof}
To see this, let $m,d\ge 0$ and observe that,
\begin{align*}
\Pro(X+Y> 2m) &=  \int_{0}^\infty \Pro(Y >  2m-x) d\mu(x)  = \int_0^\infty F(2m-x) d\mu(x)   \\
& \ge \int_{L^g_{m,d}} F(2m-x) d\mu(x)
= \int_{L^g_{m,d}}\hspace{-7pt} e^{-g(2m-x)} d\mu(x) \\
& \ge e^{-2g(m)-2d} \int_{L^g_{m,d}}\hspace{-7pt} e^{g(x)} d\mu(x)\\
& \ge e^{-2g(m)-2d} |L^g_{m,d}|_g  & \textcolor{gray}{\text{by \eqref{eq: g(L)} }} \\\\
&={ e^{-2d}\ |L^g_{m,d}|_g \ \Pro(X>m)^2  }.
\end{align*}

\end{proof}

Next we state two useful observations. The first is a relation between concavity points of two functions of bounded difference.

\begin{obs}\label{obs: domination of L}
Let $\delta>0$ and let $f_1,f_2:\R_+\to\R$ be such that $0\le f_1-f_2\le\delta$. Then for all $m,d \ge 0$ we have
$L^{f_2}_{m,d}\subseteq L^{f_1}_{m,d+\delta}$.
\end{obs}

The second regards the structure of concavity points of a convex function.

\begin{obs}\label{obs: interval}
If $f$ is convex, then for any $m,d>0$ it holds that $L^{f}_{m,d}=[m-t,m+t]$ for some $t\ge 0$.
\end{obs}
\begin{proof}
By definition $L^{f}_{m,d}$ is symmetric around $m$ and closed.
By convexity of $f$,
\[
f\left(\frac {x+y}2\right)+f\left(2m-\frac{x+y}{2}\right) \le \frac {f(x)+f(y)}{2} + \frac {f(2m-x)+f(2m-y)}{2},
\]
so that $x,y\in L^f_{m,d} \Rightarrow \frac {x+y}{2}\in L^f_{m,d}$.
Also note that, since $f$ is convex on $[0,\infty)$, it is continuous on $(0,\infty)$.
Observing that $L^f_{m,d}$ is contained in $[0,2m]$
the observation follows.
\end{proof}

\bigskip
{\bf Convex Minorant.} 
The \emph{convex minorant} of $g$, which we denote by $h$, is the maximal non-decreasing convex function such that $h(x)\le g(x)$ for all $x\ge 0$. Formally,
\[
h(x) := \sup \{ \tilde h(x) : \ \tilde{h}:\R_+\to \R_+ \text{ is convex and non-decreasing, and } \tilde{h}(t)\le g(t) \text{ for all } t\ge 0\}.
\]
As convexity and non-decreasing monotonicity are preserved by taking point-wise supremum, the function $h$ is itself convex and non-decreasing.
Notice that $h:\R_+\to\R_+$ obeys $h(0)=0$, and 
is continuous and non-decreasing (possibly equal to $0$ on an interval $[0,a]$).
Another useful property is that $h$ is an affine function
(i.e., a polynomial of degree at most $1$) on any interval where $h<g$. We end with the proof of this fact.

\begin{lem}\label{lem: h lin if not g}
Let $I\subset \R_+$ be a compact interval. If $\inf_I (g-h)>0$, then $h$ is affine on $I$.
\end{lem}

\begin{proof}
Denote $I=[x_0,x_1]$ and let $\ell(x)$ be the affine function satisfying $\ell(x_0)=h(x_0)$ and $\ell(x_1)=h(x_1)$.
Since $h$ is convex, we have either $h\equiv\ell$ on $I$, or $h<\ell$ on $(x_0,x_1)$ and $h>\ell$ outside of $[x_0,x_1]$.
Assume towards obtaining a contradiction that the latter holds.
Define
$\ell_0 = \ell +\inf_I (g-\ell)$. By definition $\inf_I (g-\ell_0)=0$, and in particular $g\ge \ell_0$ on $I$.
Note that outside $I$ we have $h > \ell \ge \ell_0$.
By maximality of $h$, we deduce that $ h\ge \ell_0$ everywhere (otherwise, $\max(h,\ell_0)$ would replace $h$ as the convex minorant of $g$).
This yields:
\[
\inf_I (g-h) \le \inf_I (g-\ell_0) =0,
\]
which contradicts our assumption.
\end{proof}

\subsection{Nearly convex case}\label{sec: main: conv}
This section is dedicated to the case of $g$ being within bounded distance from a convex function, i.e., the ``nearly convex case''.
Fix $h$ to be the convex minorant of $g$.
The main proposition of this section is the following.

\begin{prop}\label{prop: near conv}
If $\sup_{x\in\R_+} (g(x)-h(x)) <\infty$, then
$\exists d\ge 0: \  \limsup_{m\to\infty} |L^g_{m,d}|_g = \infty.$.
\end{prop}

Through Lemma~\ref{lem: reduced}, this proposition proves Theorem~\ref{thm: main}
for the nearly convex case.

Using Observation~\ref{obs: domination of L} we reduce Proposition~\ref{prop: near conv} to the following lemma.
\begin{lem}\label{lem: conv}
Let $d>0$ and let $f:\R_+\to\R_+$ be an increasing convex function with $f(0)=0$. Then:
$$\limsup_{m\to\infty} |L^f_{m,d}|_f = \infty.$$
\end{lem}

\begin{proof}[Reduction of Proposition~\ref{prop: near conv} to Lemma~\ref{lem: conv}]
Let $\delta=\sup_{x\in\R_+} (g(x)-h(x))$. We shall show that $\limsup_{m\to\infty} |L^g_{m,1+\delta}|_g = \infty$.
By Observation~\ref{obs: domination of L} we have $L^g_{m,1+\delta}\supseteq L^h_{m,1}$.
By Observation~\ref{obs: interval}, $L^h_{m,1}=[m-t,m+t]$ for some $t>0$. Thus,
\[
|L^g_{m,1+\delta}|_g \ge |L^h_{m,1}|_g \ge g(m+t)-g(m-t)\ge h(m+t)-h(m-t)-\delta=|L^h_{m,1}|_h -\delta,
\]
where the second inequality follows from \eqref{eq: g(L)}.
As $h$ is convex and increasing, Lemma~\ref{lem: conv} implies that $\limsup_{m\to\infty} |L^h_{m,1}|_h = \infty$, which together with
the last inequality concludes the reduction.
\end{proof}

It remains to prove Lemma \ref{lem: conv}.

\newcommand{\M}{3}
\newcommand{\A}{1.75}
\newcommand{\B}{4.25}
\newcommand{\DD}{0.2}

\begin{figure}[H]
\centering
\begin{tikzpicture}[scale=2.25,
	declare function={ f(\x) = 0.5 + 0.1*\x*\x;}]

     	\draw[->] (-1, 0) -- (5,0) node  (xaxis) [right] {};
        \coordinate (a) at (\A, {f(\A)} );
	\coordinate (b) at (\B, {f(\B)} );
	\coordinate (d) at  (\DD,{f(\DD)});
	\coordinate (bminus) at (\B, {2*f(\M) - f(\B) } );
	\coordinate (m) at (\M, {f(\M)} );
	\coordinate (m0) at (\M,0);
	\coordinate (ma) at (\M,{f(\A)} );
     	\draw[domain={\DD}:5, smooth, variable=\x, blue,thick] plot ({\x},  {f(\x)}) node[above] {$f$};
     	\draw[domain=-1:{\DD}, smooth, variable=\x, blue,thick] (-1,0.45) parabola (d);
     
      	\draw [very thin](a) -- (b);
     
        \coordinate (top) at (intersection of a--b and m--m0);
     
       \fill[blue] (m) circle (0.8pt);
       \fill[blue] (a) circle (0.8pt);
       \fill[blue] (b) circle (0.8pt);
        \fill[black] (d) circle (0.8pt);
        \fill[black] (top) circle (0.6pt);
      \draw[dotted] (m |- m) node[right] {} -| (xaxis -| m)  node[below] { ${m}\atop{f^{-1}(y)}$};
	 \node[below] at (\M,0) { ${m}\atop{f^{-1}(y)}$};

        \draw[dotted]  (a |- a) node[left] {} -| (xaxis -| a)  node[below] {$m - s\atop{} $};
	\draw[dotted]  (b |- b) node[right, font=\scriptsize] {$f(m+s)=y+c$} -| (xaxis -| b) node[below] {$m+s \atop f^{-1}(y+c)$};
	\draw[dotted]  (d |- d) node[left] {} -| (xaxis -|d) node[below] {${}\atop{f^{-1}(y-c)} $};

	 \draw [dotted ] (\M, {f(\A)}) -- (\B,  0.5+0.1*\A^2) coordinate (ba) node[right, font=\scriptsize, black] {$f(m-s)$};
	 \draw [dotted] (m) -- (\B,  0.5+0.1*\M^2) coordinate (bmid) node[right, font=\scriptsize] { $f(m)=y$};

 	 \draw [dotted ] (top) -- ( b |- top) coordinate (btop); 
         \draw[ForestGreen, decorate, decoration={brace,amplitude=3.5}]  (bmid)--(btop) node [midway, font= {\scriptsize}, left=1pt] {$d<$};
         \draw[ForestGreen] (top) --(m);

  	\draw  [red, decorate,decoration={brace,amplitude=5},xshift=-4pt,yshift=0pt] 
  	(b)--(bmid)
     	node [midway, right=3pt, red, font=\scriptsize] {$c$};

     	\draw  [red, decorate,decoration={brace,amplitude=5},xshift=-4pt,yshift=0pt] 
     	(m |- m)--(d -| m)
     	node [midway, right=3pt, red, font=\scriptsize] {$c$};

     	\draw  [ForestGreen, decorate,decoration={brace,mirror,amplitude=5, aspect=0.7},xshift=-4pt,yshift=0pt] 
     	(m) -- (ma) 
     	node [pos=0.7, left=3pt, ForestGreen,font=\scriptsize] {$c-2d>$};
	    	

	\draw [red, thin] (d)--(\M,{f(\DD)} )  node [midway, below=-0.2em, font=\scriptsize]{$f^{-1}(y)-f^{-1}(y-c) $};
          \draw[red, thin]  (m |- m)  -| (d -| m) ;

	\draw [ForestGreen] (a)--(ma) node [midway, below=-0.2em, font=\scriptsize] {$f^{-1}(y+c)-f^{-1}(y)$};
       \draw[ForestGreen]  ({\M-0.01}, {f(\M)} )  -- ({\M-0.01}, {f(\A)} );


	\node [draw]   at ({0.5*(\A+0.2)},2)                              
	{$\frac{ \textcolor{ForestGreen}{ f^{-1}(y+c)-f^{-1}(y) } }  	
    {\textcolor{red}{ \vphantom{\rule{1pt}{8pt}} f^{-1}(y)-f^{-1}(y-c)}} < \frac {\textcolor{ForestGreen}{c-2d} }{\textcolor{red}{ c}}$};
\end{tikzpicture}

\caption{
Illustration of the main argument in the proof of Lemma~\ref{lem: conv}. Observe that our choice of $c$ and $s$ guarantees 
that the average of $f(m+s)$ and $f(m-s)$
is greater than $f(m)$ by at least $d$. Since the difference between $f(m+s)$ and $f(m)$ is $c$, we deduce that the difference between $f(m)$ and $f(m-s)$ is at most $c-2d$. One may then use the convexity of $f$ to obtain \eqref{eq: reg}, written at the top left of the figure. 
} 
\label{fig: nearly convex}
\end{figure}
\begin{proof}[Proof of Lemma~\ref{lem: conv}]
This proof is accompanied by Figure \ref{fig: nearly convex}.
Assume towards obtaining a contradiction that there exists $c>0$ such that $|L^f_{m,d}|_f<c$ for all $m>0$.
Observe that as $f$ is convex it must be continuous, and since it is increasing it must have a
well defined inverse function $f^{-1}:\R_+\to\R_+$. For now, fix $m>0$ and let $s=s(m)>0$ be such that $f(m+s)=f(m)+c$.
Using Observation~\ref{obs: interval}, we may write $L^f_{m,d}=[m-t,m+t]$ for some $t=t(m)>0$.
By our assumption we thus have
$m+s\notin L^f_{m,d},$ or equivalently,
$f(m)+c+f(m-s)>2(f(m)+d).$
Hence,
\[f(m-s)>f(m)-c+2d.\] Writing $y=f(m)$ we get, using the monotonicity of $f$, that
\[
m-s > f^{-1}(y-c+2d).
\]
Recalling that $f(m+s)=y+c$, we obtain that
$2m > f^{-1}(y-c+2d) + f^{-1}(y+c)$, and hence
\begin{equation}\label{eq: dif}
 f^{-1}(y+c)- f^{-1}(y) < f^{-1}(y) - f^{-1}(y-c+2d), \quad \text{ for any } y>0.
\end{equation}
Observe that, since $f^{-1}$ is concave, we have:
\[
f^{-1}(y) - f^{-1}(y-c+2d) \, \le \, \frac{c-2d}{c} \left(f^{-1}(y)-f^{-1}(y-c)\right).
\]
Using this in~\eqref{eq: dif}, and denoting $q=\frac{c-2d}{c}  \in (0,1)$ for short, we obtain that
\begin{equation}\label{eq: reg}
 f^{-1}(y+c)- f^{-1}(y) < q \left( f^{-1}(y)-f^{-1}(y-c) \right) .
 \end{equation}
Applying this iteratively, we get that for any $k\in\N$:
\[
f^{-1}(kc)- f^{-1}((k-1)c) < q^{k-1} f^{-1}(c),
\]
so that $f^{-1}(Nc) = \sum_{k=1}^N (f^{-1}(kc)- f^{-1}((k-1)c)) < \frac{f^{-1}(c)}{1-q}$ for any $N\in\N$.
We conclude that $f^{-1}$ is bounded, and hence that $f$ is compactly supported, in contradiction with our assumption.
\end{proof}

\subsection{Nearly concave case}\label{sec: main: conc}

This section is dedicated to prove Theorem~\ref{thm: main} in the nearly-concave case, i.e., the case when
$\lim_{x\to \infty} (g(x)-h(x))=\infty$. This is done through the following proposition.

\begin{prop}\label{prop: far conv}
If $\lim_{x\to \infty} (g(x)-h(x))=\infty$, then $X$ and $Y$ satisfy \eqref{eq: main}.
\end{prop}

\begin{proof}
Denoting $f(x)=g(x)-h(x)$, our assumption is that $\lim_{x\to\infty} f(x)=\infty$.
By Lemma~\ref{lem: h lin if not g}, this implies that $h$ is affine on some infinite ray $[m_0,\infty)$,
and hence there exists $b\in\R$ such that
\[
2g(m)-g(2m) = 2f(m)-f(2m)+b, \quad \forall m>m_0.
\]
Observe that: 
\begin{equation*}
\frac{\Pro(X+Y > 2m)} {\Pro(X> m)^2} \ge \frac{\Pro(X > 2m)}{\Pro(X > m)^2} =e^{2g(m)-g(2m)}=e^{2f(m)-f(2m)+b},
\end{equation*}
for $m>m_0$.
Therefore, in order to prove~\eqref{eq: main} it is enough to show that
\begin{equation}\label{eq: h=0 goal}
\limsup_{m\to\infty} \left( 2f(m)-f(2m) \right) = \infty.
\end{equation}
Assume to the contrary that \eqref{eq: h=0 goal} does not hold. Then there exists some $c>0$ such that for all large enough $m$:
\begin{equation}\label{eq: rec}
f(2m) \ge 2f(m)-c.
\end{equation}

Since $\lim_{x\to\infty} f(x) =\infty$, there exists $a>m_0$ such that $f > 2c$ on $[a,\infty)$.
Let $x>2a$. There exists a unique $k\in \N$ such that $x_0 = \frac x{2^k} \in [a,2a)$.
By repeatedly using~\eqref{eq: rec}, we get that
\begin{align*}
f(x)=f(2^k x_0)\ge 2^k f(x_0)-(2^k-1)c \ge 2^k c \ge \frac x a  c.
\end{align*}
This implies
\[
f(x) \ge \max\left(0, \frac c{a}(x-2a)\right), \quad \forall x>0,
\]
which in turn yields $g(x) \ge h(x) +  \max(0, \frac c{a}(x-2a))$, in contradiction with the fact that $h$ is the largest convex function which is less-equal to $g$.
Thus~\eqref{eq: h=0 goal} holds, and we are done.
\end{proof}

Notice that, in the course of proving Proposition~\ref{prop: far conv}, we showed the following
\begin{cor}\label{cor: far conv}
If $\lim_{x\to \infty} (g(x)-h(x))=\infty$, then $\limsup_{m\to\infty} \left( 2g(m)-g(2m) \right) = \infty.$
\end{cor}
This will be of use in the non-i.i.d. case.

\subsection{Oscillating case}\label{sec: main: osci}
In this section we consider the case in which the distance between $g$ and its convex minorant $h$ oscillates.
The main statement of this section is the following.
\begin{prop}\label{prop: oscillating}
If
\[
\liminf_{x\to \infty} (g(x)-h(x))<\infty \quad\quad \text{ and  }\quad \quad\limsup_{x\to \infty} (g(x)-h(x))=\infty,
\]
Then $\exists d\ge 0: \quad  \limsup_{m\to\infty} |L^g_{m,d}|_g = \infty.$
\end{prop}
This case holds, for instance, for the function $g(x) =\lceil  \sqrt{x} \rceil ^2 $, adjusted at points of discontinuity to be right-continuous.
Through Lemma~\ref{lem: reduced}, the proposition would imply that Theorem~\ref{thm: main} holds in the oscillating case.

\renewcommand{\A}{1.75}
\renewcommand{\B}{4.25}
\renewcommand{\M}{2.5}
\newcommand{\D}{4.8}

\begin{figure}[H]
\centering
\begin{tikzpicture}[scale=2]
	

 	\draw[->] (1, 0) -- (5,0) node  (xaxis){};
        \coordinate (a-) at (\A, {0.25*\A+0.1});
      	\coordinate (a) at  ({\A+0.05}, {0.25*\A+0.01});
      	\coordinate (ar) at  ({\A+0.38}, {0.25*\A+0.8});
  	\coordinate (aep) at  ({\A+0.28}, {0.25*\A+0.4});
	\coordinate (b) at (\B,  {0.51*\B});
	\coordinate (blow) at (\B,  {0.5*\B-0.1});
	\coordinate (m-) at (\M,   {0.4*\B});
	\coordinate (m) at ({\M-0.11},   {0.4*\B-0.035});
	\coordinate (aept) at ({2*(\M-0.11) - (\A+0.28) }, 1.75);
	\coordinate (d) at (\D,   {0.67*\B});
	\coordinate (f) at (1.3,0);
	\fill[blue] (ar) circle (0.8pt);
	\fill[blue] (b) circle (0.8pt);
	\fill[blue] (m) circle (0.8pt);
	\fill[ForestGreen] (aept) circle (1pt);
	\fill[ForestGreen] (aep) circle (1pt);

	\draw [blue, thick] (1,0.3) to [in=190, out=50] (a-) to [in=190, out=20] (m-) to [in=190, out=10] (b) to [in=190, out=30]  (d)
	node[below] {$g$}; 
	\draw [red, thick, shorten >=-4.2em, shorten <=-3.3em] (a)--(blow) node [right = 3.4em, above=0.7em] {$h$};
        \draw[dashed]  (ar |- ar) node[left] {} -| (xaxis -| ar)  node[below] {$a$};
        \draw[dashed,ForestGreen]  (aep |- aep) node[left] {} -| (xaxis -| aep) coordinate (aepx) node[above=0.3em, left=0.1em, ForestGreen, font=\scriptsize] {$a-\ep$};
        \draw[dashed,ForestGreen]  (aept |- aept) node[left] {} -| (xaxis -| aept) coordinate (aeptx) node[above=0.3em, right=0.1em, ForestGreen, font=\scriptsize] {$2m-(a-\ep)$};
             \fill[ForestGreen] (aepx) circle (0.8pt);
              \fill[ForestGreen] (aeptx) circle (0.8pt);
        \draw[dashed]  (b |- b) node[left] {} -| (xaxis -| b)  node[below] {$b$} coordinate (bx);
        \draw[dashed]  (m |- m) node[left] {} -| (xaxis -| m)  node[below] {$m$} coordinate (mx);
        \draw [blue,dashed] (ar) -- (f |- ar) coordinate (af) node [left, font=\scriptsize]{$g(a)$};
        \draw [blue,dashed] (m) -- (f |- m) coordinate (mf) node [left, font=\scriptsize]{$g(m)$};
        \draw [black, dashed, ForestGreen] (aep) -- (f |- aep) coordinate (aepf) node [left, font=\tiny, ForestGreen]{$g(a-\ep)$};

	\draw[ForestGreen, ultra thick](aepx)--(aeptx);

	\draw  [black, decorate, decoration={brace, amplitude=5},xshift=-4pt,yshift=0pt] (af) -- (mf) node [midway, left=3pt, black, font=\tiny] {$\eta \ge $};
	
	\draw  [ForestGreen, decorate, decoration={brace, mirror, amplitude=5, aspect=0.25},xshift=-4pt, yshift=0pt] (aepf) -- (mf) node [
	pos=0.25,
	right=4pt, ForestGreen, font=\tiny] {$\ge \eta $};

	\draw[ForestGreen, thin] (aep) -- (aept);
	
	\coordinate (mc) at (intersection of a--blow and m--mx);
	\fill[red] (mc) circle (0.7pt);
		
	\draw  [gray, ultra thin, decorate, decoration={brace, amplitude=5, aspect=0.6},xshift=-4pt,yshift=0pt] (m) -- (mc) node [pos=0.6, right=3pt, font=\tiny, gray] {$\ge 2\eta $};
        
        \coordinate (bc) at (intersection of a--blow and b--bx);
        \fill[red] (bc) circle (0.7pt);

	\draw  [black, ultra thin, decorate, decoration={brace, mirror, amplitude=3 }, xshift=-4pt,yshift=0pt] (b) -- (bc) node [pos=0.75, left=0.5pt, font=\tiny, gray] {$ \eta\ge $};

\end{tikzpicture}

\caption{
Illustration of the main argument in the proof of Proposition~\ref{prop: oscillating}. Our choice of $b$ and $m$ guarantees that
$g(b)-h(b)<\eta$ and $g(m)-h(m)\ge 2\eta$. Our choice of $a$ guarantees that $g(m)-g(a)\le \eta$ and $g(m)-g(a-\ep)\ge \eta$ for all $\ep>0$. 
For any $x\in [a,m]$ we show that $2m-x\in [m,b)$ and $x\in L^g_{m,0}$. Then, given $d>0$, we choose $\ep>0$ sufficiently small so that $ [a-\ep,m] \subseteq L^g_{m,d}$. The fact that $|[a-\ep, m ] |_g \ge \eta$ concludes the argument.
} 
\label{fig: oscillating}
\end{figure}

\begin{proof}[Proof of Proposition~\ref{prop: oscillating}]

The proof is accompanied by Figure~\ref{fig: oscillating}.
We shall show, in fact, that $\limsup_{m\to\infty} |L^g_{m,d}|_g = \infty$ for any $d>0$ (it holds even for $d=0$, but for simplicity we do not extend the proof to this case).

Let $d,\eta>0$, we must show that there exists $m$ for which $|L^g_{m,d}|_g\ge \eta$.
As before, write $f=g-h$. Since $g$ is right-continuous and non-decreasing, it is also upper-semicontinuous,
i.e., $\limsup_{x\to x_0, \ x\ne x_0}g(x)\le g(x_0)$.
Since $h$ is continuous, we get that $f$ is upper-semicontinuous as well.
By the second premise of the proposition, there exists $m_1>0$ such that $g(m_1)-h(m_1)> 2\eta$. Define
\[
 b=\inf\{x>m_1: \, f(x)\le \eta\}, \quad\ \text{ and }\quad m = \arg \max_{[0,b]} f.
\]
Note that $b$ is well-defined due to the first premise of the proposition (provided $\eta$ is large enough), and $m$ is well-defined due to the upper-semicontinuity of $f$.
Next, define
\[a = \sup\{ 0\le x<m: \, g(x)<g(m)-\eta \}.\]
In particular, for any $\ep>0$ we have $| [a-\ep,m] |_g= g(m)-g(a-\ep)  \ge \eta$.
It remains to show that
\begin{equation}\label{eq: in L}
\exists \ep>0:  \quad [a-\ep, m]\subseteq L^{g}_{m,d},
\end{equation}
as this would imply that $|L^{g}_{m,d}|_g \ge |[a-\ep,m]|_g\ge \eta$, and complete the proof.

First, we observe that $f>\eta$ on $(a,b)$. Therefore, Lemma~\ref{lem: h lin if not g} implies that
$h$ is affine on $[a,b]$. Next, we claim that
\begin{equation}\label{eq: in interval}
\forall x \in [a,m]: \; 2m -x \in [m,b).
\end{equation}
Since $x \le m$, we have $2m-x \ge m $.
Using the affinity of $h$, we may rewrite the claimed upper bound:
$
2m-x < b \iff
m-x < b-m \iff
h(m)-h(x) < h(b)-h(m).
$
The latter holds true by the following argument:
\begin{align*}
  h(m)-h(x) &\le g(m) - g(x)  &  \textcolor{gray}{\text{by maximality of } f(m)} \\
  &    \le g(m)-g(a)         &  \textcolor{gray}{a \le x } \\
  & \le \eta                  &  \textcolor{gray}{\text{definition of } a}  \\
  &  < g(m)-h(m)-\eta                   &  \textcolor{gray}{ f(m)\ge f(m_1)>2\eta }\\
  &  \le g(b)-h(m)-\eta                      &  \textcolor{gray}{ m\le b } \\
  & \le h(b)-h(m).              &  \textcolor{gray}{\text{definition of } b }
\end{align*}

Now that~\eqref{eq: in interval} is established, we may use the affinity of $h$ in $[a,b]$ to obtain
$$h(m)-h(x)=h(2m-x)-h(m),$$ for any $x\in [a, m]$.
Using continuity of $h$, we choose $\ep>0$ small enough so that
\[
\forall x\in[a-\ep,m]:  \quad 2m-x \in [m,b), \quad \text{ and } \quad h(m)-h(x)\ge h(2m-x)-h(m)-2d.
\]
These two facts, combined with the definition of $m$, yield that for any $x \in [a-\ep,m]$,
\begin{align*}
  g(m)-g(x) &\ge h(m)-h(x) \ge h(2m-x) - h(m) -2d
\ge g(2m-x)-g(m)-2d.
\end{align*}
Rearranging this inequality yields $g(x)+g(2m-x) \le 2 g(m) +2d$ for any $x\in [a-\ep, m]$, which, in turn, yields \eqref{eq: in L}. The Proposition follows.
\end{proof}


\section{Proof of Theorem~\ref{thm: main}: non i.i.d. case}\label{sec: main non iid}
In this section we prove Theorem~\ref{thm: main} by reducing it to the i.i.d. case which was tackled in the previous section.

Let $X$, $Y$ be independent random variables on $\R_+$, and denote
\[
g_0(x):=-\log \Pro(X > x), \quad g_1(x):=-\log \Pro(Y > x), \quad g(x):= \frac 1 2 (g_0(x) + g_1(x)).
\]
Define the functions
\[
p^0_m(x) = \frac {g_0(x)+g_1(2m-x)}{2}, \quad
p^1_m(x) = \frac {g_1(x)+g_0(2m-x)}{2},
\]
and the sets
\begin{align*}
&L^{0}_{m,d} = \left\{ \ell\in [0,2m] \, : \quad p^0_m(\ell)\le g(m) +d  \right\}, \\
&L^{1}_{m,d} = \left\{ \ell\in [0,2m] \, : \quad p^1_m(\ell)\le g(m) +d  \right\}.
\end{align*}

The following lemma is a generalization of Lemma~\ref{lem: reduced} for the non-i.i.d. case.
\begin{lem}\label{lem: reduced not iid}
If\ \  $\exists j\in\{0,1\}, d\ge 0:\ \limsup_{m\to\infty} \big|L^{j}_{m,d} \big|_{g_j} = \infty,$   then $X$ and $Y$ satisfy \eqref{eq: main}.
\end{lem}

\begin{proof}
Similar to the proof of Lemma~\ref{lem: reduced}, we observe that:
\begin{align*}
\Pro(X+Y> 2m) &= \int_0^\infty e^{-g_1(2m-x)} d\mu_0(x)
 \ge e^{-g_0(m)-g_1(m)-2d} \int_{L^0_{m,d}} e^{g_0(x)} d\mu_0(x).
\end{align*}
By~\eqref{eq: g(L)}, we conclude that
\[
\Pro(X+Y> 2m) \ge e^{-2d}\Pro(X>m)\Pro(Y>m)\cdot |L^0_{m,d}|_{g_0}.
\]
The same holds after interchanging the roles of $X$ and $Y$, and so the lemma follows.
\end{proof}

\begin{proof}[Proof of Theorem~\ref{thm: main}]
Let $h$ be the convex minorant of $g=(g_0+g_1)/2$. We first deal with the case $\lim_{x\to\infty} (g(x)-h(x))=\infty$. Notice that:
\[
\frac{\Pro(X+Y > 2m)} {\Pro(X> m)\Pro(Y>m)} \ge \frac{\sqrt{ \Pro(X > 2m)  \Pro(Y>2m)} }{\Pro(X > m)\Pro(Y>m)} =e^{-\frac 1{2}(g_0(2m)+g_1(2m)) + g_0(m)+g_1(m)} = e^{2g(m)-g(2m)}.
\]
This is unbounded in our case by Corollary~\ref{cor: far conv}, and therefore~\eqref{eq: main} holds.
In the remaining cases, we know by Propositions ~\ref{prop: near conv} and~\ref{prop: oscillating} that
\[
\exists d>0:  \quad \limsup_{m\to\infty} \big| L^g_{m,d} \big|_g =\infty.
\]
Since $|A|_g = \frac 1 2\left(|A|_{g_0} + |A|_{g_1}\right)$ for any Borel set $A\subset\R$, we can assume without loss of generality that
\begin{equation}\label{eq: big g0}
\limsup_{m\to\infty} \big| L^g_{m,d} \big|_{g_0} =\infty.
\end{equation}

Define:
\[
\beta:=\sup \left\{g(m)-p^1_m(x)\ :\ m>0, x\in L^g_{m,d} \right\}
\]

First suppose that $\beta<\infty$. In this case, for every $m>0$ and every $x\in L^g_{m,d}$ we have
$g(m)- p_m^1(x)\le \beta$. In the same time, by definition of $L^g_{m,d}$, we also have
\[
p^0_m(x)+p^1_m(x)=g(x)+g(2m-x) \le 2(g(m)+d).
\]
These two things together yield:
\[
p^0_m(x) \le 2g(m) +2d-p^1_m(x) \le g(m)+2d+\beta.
\]
By definition, this means that $L^g_{m,d} \subseteq L^0_{m,2d+\beta}$. By~\eqref{eq: big g0} this yields
$\limsup_{m\to\infty}|L^0_{m,2d+\beta}|_{g_0}=\infty$, thus by Lemma~\ref{lem: reduced not iid} we are done.

We are left with the case of $\beta=\infty$. Fix a large number $\eta>0$. Since $\beta=\infty$, there exists $m>0$ and
$ x\in L^g_{m,d}$ such that
\begin{equation}\label{eq: dif p1}
p^1_m(m)-p^1_m(x)=g(m)-p^1_m(x)>\eta.
\end{equation}
Assume first that $x<m$.
Define
\[
s = \inf\{y>x: \: p^1_m(y)-p^1_m(x)>\eta \}.
\]
Notice that $s$ is well-defined and $x<s\le m$.
We shall show that:
\begin{enumerate}[(I)]
\item \label{it: big g1} $|[x,s]|_{g_1} \ge 2\eta$.
\item \label{it: in L1} $[x,s] \subseteq L^1_{m,0}$.
\end{enumerate}
These two items will imply that $|L^1_{m,0}|_{g_1} \ge 2\eta$, thus by Lemma \ref{lem: reduced not iid} our proof would be complete.
For item~\eqref{it: big g1}, notice that by upper-semi-continuity of $p^1_m$ we have
\[
\eta\le p^1_m(s)-p^1_m(x) = \frac{1}{2} (g_1(s)-g_1(x) ) +\frac 1 2 (g_0(2m-s)-g_0(2m-x)) \le \frac 1 2 (g_1(s)-g_1(x) ),
\]
which means that $|[x,s]|_{g_1}\ge 2\eta$.
For item~\eqref{it: in L1}, observe that for any $y\in [x,s]$ we have (using the definition of $s$ and~\eqref{eq: dif p1}):
\[
p^1_m(y) \le p^1_m(x)+ \eta \le p^1_m(m).
\]
Note that to get this inequality at $y=s$ we used also right-continuity of $p^1_m(\cdot)$. Thus, by definition, $y\in L^1_{m,0}$ as required.

The case $x>m$ follows similarly. We write $\tilde{x}=2m-x$ and observe that $\tilde{x}<m$. Hence \eqref{eq: dif p1} becomes $p^0_m(m)-p^0_m(\tilde x)>\eta$, so one may replace $x$ by $\tilde x$ and $p^1$ by $p^0$ in the previous argument to obtain $|L^0_{m,0}|_{g_0}>2\eta$ and end similarly by Lemma~\ref{lem: reduced not iid}.
\end{proof}

\section{Proof of Theorem~\ref{thm: general}: high dimensions and weighted averages}\label{sec: general}

\subsection{Preliminaries}
{\bf Notation.}
As before, we write $g(x) := -\log \mu((x,\infty))$, and
let $h$ be the convex minorant of $g$. 
We regard $g$ as a function from $\R$ to $\R$, setting $g(x)=0$ for $x< 0$.
We also fix $n\in\N$, and let $X_1,\dots,X_n$ be i.i.d. random variables, each distributed with law $\mu$.
We denote by $g^{\PLH n}$ the measure in $\R^n$ which is the product of $n$ copies of the one-dimensional measure defined by $g$. For a Borel set $A\subseteq \R^n$, we write $|A|_{g^{\PLH n}}$ for the measure of $A$ under this product. Finally, we write $\Lambda_n=\{(\lm_1,\dots,\lm_n)\in(0,\infty)^n : \: \sum_j \lm_j=1\}$ and fix $\bar\lm = (\lm_1,\dots,\lm_n)\in\Lambda_n$.
For $\bar x=(x_1,\dots, x_{n-1})\in\R_+^{n-1}$ and $m\in \R_+$, write
\[\phi_{\bar \lm}(\bar x,m)=\frac1{\lm_n}\left( m- \sum_{j=1}^{n-1} \lm_j x_j\right).\]

{\bf The set of $m$-symmetric $d$-concavity points.} We will use the following generalization of $L^f_{m,d}$.
For $f:\R\to \R$ and $m, d\ge 0$, we define the set
\[
L^f_{m,d, \bar \lm}= \left\{  (x_1,\dots, x_{n-1})\in\R_+^{n-1}\, \bigg| \;  \sum_{j=1}^{n-1} f( x_j) + f(x') \le n f(m) +n d   \right\},
\]
where $x'=\phi_{\bar \lm}(\bar x,m)$.


We end the preliminaries by generalizing Lemma~\ref{lem: reduced} to the high-dimensional setting.
\begin{lem}\label{lem: gen reduced}
If\ \  $\exists d\ge 0:\ \limsup_{m\to\infty} \left|L^g_{m,d,\bar\lm}\right|_{g^{\PLH (n-1)}} = \infty,$
  then $X_1,\dots,X_n$ satisfy \eqref{eq: gen}.
\end{lem}

\begin{proof}
Writing $L=L^g_{m,d,\bar \lm}$ for short, we have:
\begin{align*}
\Pro &\left( \sum_{j=1}^n \lm_j X_j  > m \right) = \int_{\R_+^{n-1}} \Pro\big(X_n>\phi_{\lm}(\bar x,m)\big)\ d\mu(x_1)\dots d\mu(x_{n-1}) & \textcolor{gray}{\text{definition of }  \phi_{\lm}(\bar x,m) }\\
& \ge \int_L e^{-g(\phi_{\lm}(\bar x,m))} d\mu(x_1)\dots d\mu(x_{n-1}) & \textcolor{gray}{ F = e^{-g}, \text{restrict to } L } \\
& \ge e^{-nd -ng(m)} \int_L e^{g(x_1)+\dots +g(x_{n-1}) } d\mu(x_1)\dots d\mu(x_{n-1}). &  \textcolor{gray}{\text{definition of } L }
\end{align*}

Observing that $\Pro\big(X_1>m\big)^n= e^{-n g(m)}$, we conclude that
\[
\frac{\Pro(\sum_j \lm_j X_j > m) }{ \Pro(X_1>m)^n } \ge e^{-nd}  \int_L e^{g(x_1)+\dots +g(x_{n-1})} d\mu(x_1)\dots d\mu(x_{n-1}) = e^{-nd} \left| L \right|_{g^{\PLH (n-1)}}.
\]

The last equality follows from the definition of product measure, and the fact that in one-dimension $\int_A e^g d\mu \le  |A|_g$ (by~\eqref{eq: g(L)}).
The lemma follows.
\end{proof}

\subsection{Nearly convex case}
As in the proof of Theorem~\ref{thm: main}, we first treat the case where $g$ is closely approximated by its
convex minorant. Our goal is to prove the following generalization of Proposition~\ref{prop: near conv}, which together with Lemma
 \ref{lem: gen reduced} implies Theorem \ref{thm: general} in the nearly convex case.
\begin{prop}\label{prop: near conv gen}
If $\sup_{\R_+} (g-h)<\infty$ and 
$\lm_1\ge \dots \ge \lm_{n-1}\ge \lm_n$, then there exists $d\ge 0$ such that
$\limsup_{m\to\infty} \left|L^g_{m,d,\bar\lm}\right|_{g^{\PLH (n-1)}} = \infty$.
\end{prop}

We begin with a few simple observations, omitting the proofs when these are straightforward.
First, we generalize Observations~\ref{obs: domination of L} and~\ref{obs: interval}.

\begin{obs}\label{obs: domination of L gen}
Let $\delta>0$ and let $f_1,f_2:\R_+\to\R$ be such that $0\le f_1-f_2 \le \delta$. Then
for all $m,d\ge 0$ we have:
$L^{f_1}_{m,d+\delta,\bar\lm}\supseteq L^{f_2}_{m,d,\bar\lm}$.
\end{obs}

\begin{obs}\label{obs: interval gen}
If $f$ is convex, then for any $m,d\ge 0$ the set $L^f_{m,d,\bar\lm}$ is convex.
\end{obs}

Next we observe a simple inclusion in the case $n=2$.
\begin{obs}\label{obs: 2 dim}
Let $f:\R_+\to\R$ be non-decreasing. Then for all $\al \in (\frac 1 2,1)$ and $m,d\ge 0$ we have
\[ L^f_{m,d,(\frac 1 2,\frac 1 2)} \cap [m,2m] \subseteq  L^f_{m,d,(\al,1-\al)}.\]
\end{obs}

The next observation relates concavity points of any dimension $n$ to certain concavity points of dimension $2$.
Denote for short $(\lm^*_{n-1},\lm^*_n) =\left(\frac {\lm_{n-1}}{\lm_{n-1}+\lm_{n}},\frac {\lm_n}{\lm_{n-1}+\lm_{n}}\right)$.

\begin{obs}\label{obs: to low dim}
Let $f:\R_+\to\R_+$ be non-decreasing. Then for any $m ,d, x \ge 0$ we have
\[
(m,\dots,m,x) \in L^f_{m,d,(\lm_1,\dots,\lm_n)} \iff
x \in L^f_{m, \frac{nd}{2}, \two }.
\]
\end{obs}
\begin{proof}
This could be directly derived from the fact that
\[
\phi_{\bar \lm}((m,\dots,m,x),m)=\phi_{\two}(x,m).\]
\end{proof}  

Our last observation concerns with changing one coordinate of a
concavity point.

\begin{obs}\label{obs: box}
Let $m,d,c\ge 0$, $k\in\{1,\dots,n\}$, $f:\R_+\to\R_+$ non-decreasing
and a point $(p_1,\dots,p_{n-1})\in L^f_{m,d,\bar \lm}$. Then we have $(p_1,\dots,p_{k-1},q_k,p_{k+1},\dots,p_{n-1})\in L^f_{m,d+\frac c n,\bar \lm}$ for
all $q_k>p_k$ such that $f(q_k)-f(p_k)\le c$.
\end{obs}

\begin{proof}
Denote $\bar x=(p_1,\dots,p_{n-1})$, $\bar y=(p_1,\dots,p_{k-1},q_k,p_{k+1},\dots, p_{n-1})$ and write $x'=\phi_{\lm}(\bar x,m)$ and  $y'=\phi_{\lm}(\bar y,m)$.
Then
\begin{align*}
&\sum_{j\ne k} f(p_j) + f(q_k) + f(y')  \\
 & \quad=  \sum_{j<n} f(p_j)+f(x') + ( f(q_k)-f(p_k) )+ (f(y')-f(x') ) \\
& \quad \le n (f(m)+d) + c + 0,
\end{align*}
where for the last inequality we used the definitions of $L^f_{m,d,\bar\lm}$, our assumption $f(q_k)-f(p_k)\le c$, and monotonicity of $f$ applied to the fact that $y'\le x'$ (the latter holds since $\bar x$ and $\bar y$ differ only by one coordinate in which $\bar y$ is bigger). We conclude that $\bar y\in L^g_{m,d+\frac c n,\bar\lm}$, as required.
\end{proof}

We are now in position to show Proposition~\ref{prop: near conv gen}.
\begin{proof}[Proof of Proposition~\ref{prop: near conv gen}]
Fix $d>0$ and denote $\delta:=\sup_{\R_+} (g-h)$. By Observation~\ref{obs: domination of L gen}, we have
\begin{equation}\label{eq: Lhg}
L^h_{m,d+2\delta,\bar\lm}\subseteq L^g_{m,d+3\delta, \bar\lm}.
\end{equation}
We will show that the left-hand-side set is large. By Observation~\ref{obs: to low dim} applied to $h$, we have
\[
\{m\}^{n-2}\times L^h_{m,d,\two }\subseteq 
\{m\}^{n-2}\times L^h_{m,\frac{nd}{2},\two } \subseteq L^h_{m,d,\bar\lm}.
\]
Let $s>m$ be such that $h(s)-h(m)=2\delta$ (such $s$ exists since $h$ is continuous and $\lim_{x\to\infty}h(x)=\infty$). Applying Observation~\ref{obs: box} iteratively on each of the first $n-2$ coordinates of
the left-hand-side yields
\[
[m,s]^{n-2} \times L^h_{m,d,\two} \subseteq L^h_{m,d+\frac {n-2}{n}\cdot 2\delta, \bar \lm} \subseteq L^h_{m,d+2\delta,\bar\lm}.
\]
By Observation~\ref{obs: 2 dim} this implies
\begin{equation}\label{eq: inc}
[m,s]^{n-2} \times \left(L^h_{m,d,(\frac 12,\frac 12)} \cap[m,2m] \right) \subseteq L^h_{m,d+2\delta, \bar \lm}.
\end{equation}

Recall that by Observation~\ref{obs: interval} we have $L^h_{m,d,(\frac 12,\frac 12)}=[m-t_m,m+t_m]$ for some $t_m\ge0$, and by Lemma~\ref{lem: conv} combined with convexity of $h$, we have
\begin{equation}\label{eq: to inf}
\limsup_{m\to\infty} (h(m+t_m)-h(m) )=\infty.
\end{equation}
Taking $|\cdot|_{g^{\PLH (n-1)}}$ on the inclusion in~\eqref{eq: inc} yields
\begin{align*}
\big| L^h_{m,d+2\delta,\bar \lm} & \big|_{g^{\PLH(n-1)}}
\ge | [m,s] |^{n-2}_g \cdot \big|[m,m+t_m]|_g \\
& \ge \big(|[m,s]|_h-\delta \big)^{n-2} \cdot \big( | [m, m+t_m]|_h-\delta\big)\\
& \ge \delta^{n-2}  \cdot \big(h(m+t_m)-h(m)-\delta\big).
\end{align*}

Combining this with~\eqref{eq: Lhg} and \eqref{eq: to inf}, we conclude that
$ \limsup_{m\to\infty}|L^g_{m,d+3\delta,\bar\lm}|_{g^{\PLH(n-1)}} =\infty$, uniformly in $\bar\lm\in\Lambda_n$
(and, in fact, uniformly in $\cup_{n\ge 2} \Lambda_n$).
\end{proof}

\subsection{Nearly concave case}
In this case we can prove Theorem~\ref{thm: general} directly. This is a simple generalization of Proposition~\ref{prop: far conv}.

\begin{prop}\label{prop: far conv gen}
If $\lim_{x\to\infty} (g(x)-h(x) )=\infty$, then $X_1,\dots,X_n$ satisfy~\eqref{eq: gen}.
\end{prop}

\begin{proof}
Without loss of generality, assume $\lm_1=\max(\lm_1,\dots,\lm_n)$. Observe that:
\begin{align*}
\frac{\Pro(\sum_j \lm_j X_j > m)} {\Pro(X> m)^n} \ge \frac{ \Pro(\lm_1 X_1 > m) }{\Pro(X > m)^n}
\ge  \frac {F\left( m / \lm_1 \right)} {F(m)^n}
& =e^{ng(m)-g(m/\lm_1)} \\
&\ge e^{ng(m)-g(n m)}.
\end{align*}
Therefore, in order to prove Theorem~\ref{thm: general} it is enough to show that
\begin{equation}
\limsup_{m\to\infty} \left( ng(m)-g(n m) \right) = \infty.
\end{equation}
This is achieved by the same proof of Corollary~\ref{cor: far conv}, simply by replacing all the appearances of the number $2$ by $n$.
The uniformity in $\bar\lm\in\Lambda_n$ as stated in Theorem~\ref{thm: general} is clear.
\end{proof}

\subsection{Oscillating case}
We are left with the case of unbounded oscillating distance from the convex minorant.
The following proposition (which generalizes Proposition~\ref{prop: oscillating}), together with Lemma~\ref{lem: gen reduced}, would imply Theorem~\ref{thm: general} in this case.
\begin{prop}\label{prop: oscillating gen}
If $\lm_n=\max\{\lm_j: 1\le j\le n\}$, 
\[
\liminf_{x\to \infty}\  (g(x)-h(x))<\infty \quad\quad \text{ and  }\quad \quad\limsup_{x\to \infty}\ (g(x)-h(x))=\infty,
\]
then $\exists d\ge 0:\ \limsup_{m\to\infty} \left|L^g_{m,d,\bar\lm}\right|_{g^{\PLH (n-1)}} = \infty$.
\end{prop}

The following observation will be useful in the proof.
\begin{obs}\label{obs: lin}
Let $I$ be an interval on which $h$ is affine, and let $m=\arg\max_I (g-h)$. If $x_1,\dots,x_n\in I$ are such that $\sum_{j=1}^n \lm_j x_j=m$ and $\frac 1 n\sum_{j=1}^n x_j \le m$,
then $(x_1,\dots,x_{n-1})\in L^g_{m,0,\bar\lm}$.
\end{obs}

\begin{proof}
By the premise, $\sum_{j\le n} (x_j-m) \le 0$. Since $m,x_1,\dots,x_n\in I$ and $h$ is affine on $I$, we get $\sum_{j\le n}(h(x_j)-h(m)) \le 0$. By maximality of $m$, we have $g(x_j)-g(m)\le h(x_j)-h(m)$ for all $j\le n$. This imples $\sum_{j\le n} (g(x_j)-g(m))\le 0$, and since $x_n=\phi_{\bar\lm}((x_1,\dots,x_{n-1}),m)$ this implies $(x_1,\dots,x_{n-1})\in L^g_{m,0,\bar\lm}$.
\end{proof}

We now present the proof of Proposition~\ref{prop: oscillating gen}.
\begin{proof}[Proof of Proposition~\ref{prop: oscillating gen}]
Notice that our presmises guarantee that $h$ is not identically zero, so that $h$ is strictly increasing on $[m_0,\infty)$ for some $m_0>0$.
Write $f=g-h$, and fix a large $\eta>0$.
By our assumption that  $\limsup_{x\to \infty}\ (g(x)-h(x))=\infty$, there exists $m_1>m_0$ such that $g(m_1)-h(m_1)> n\eta$. 
Define
\[
 b=\inf\{x>m_1: \, f(x)\le \eta\}, \quad \text{ and }\quad m = \arg \max_{[m_1,b]} f.
\]
Note that $b$ is well-defined due to the first premise (provided $\eta$ is large enough), and $m$ is well-defined due to upper-semicontinuity of $f$.
Our goal is to show that $\big| L^{g}_{m,0,\bar\lm} \big|_{g^{\PLH (n-1)}} \ge \eta$ for all $\bar\lm\in\Lambda_n$. Notice that, since the point $m$ depends on $n$ but not on $\bar\lm\in\Lambda_n$, this will establish also the uniformity stated in Theorem~\ref{thm: general}.

Define $\Delta>0$ through the relation
\[\quad h(m+\Delta)-h(m)=\eta.\]
We will now show that
\begin{equation}\label{eq: <b}
m+(n-1)\Delta\le b.
\end{equation}
Since $g-h>\eta$ on $(m_1,b)$, by Lemma~\ref{lem: h lin if not g} there is a strictly increasing affine  function $\ell(x)$ such that $\ell(x)=h(x)$ for $x\in (m_1,b)$.
We have:
\begin{align*}
\ell(b)=h(b)&\ge g(b)-\eta  & \textcolor{gray}{\text{definition of }b} \\
  & \ge g(m)-\eta         &  \textcolor{gray}{m \le b } \\
  & > \ell(m)+(n-1)\eta   &  \textcolor{gray}{(g-\ell)(m)=f(m)>n\eta} \\
  & = \ell(m+(n-1)\Delta).  & \textcolor{gray}{\text{definitions of } \Delta, \ell}
\end{align*}
Since $\ell$ is strictly increasing, this proves~\eqref{eq: <b}.
As a consequence, we conclude that
\begin{equation}\label{eq: dual<b}
\bar x\in [m-\Delta,m]^{n-1}\quad  \Longrightarrow \quad \phi_{\bar\lm}(\bar x,m)\in [m,b].
\end{equation}

Also notice that since $\lm_n\ge \frac 1 n$ we have:
\begin{equation}\label{eq: av<m}
\bar x\in [m-\Delta,m]^{n-1}, \sum_{1\le j\le n} \lm_j x_j =m \quad  \Longrightarrow \quad
\frac 1 n \sum_{1\le j\le n} x_j\le m.
\end{equation}

To conclude the proof, consider two cases. In the first case, $f>\eta$ on $(m-\Delta,m)$. Then, by Lemma~\ref{lem: h lin if not g}, $h$ is affine on $I:=[m-\Delta,b]$.
This, together with ~\eqref{eq: dual<b} and~\eqref{eq: av<m}, fulfill the conditions of Observation~\ref{obs: lin} and we conclude that $[m-\Delta,m]^{n-1}\subseteq L^g_{m,0,\bar\lm}$. On the other hand, by maximilaty of $f(m)$ we have
$ g(m)-g(m-\Delta) \ge h(m)-h(m-\Delta)=h(m+\Delta)-h(m) = \eta.$
Therefore,
\[
\big| L^{g}_{m,0,\bar\lm} \big|_{g^{\PLH (n-1)}} \ge |[m-\Delta,m]|_g^{n-1} \ge \eta^{n-1}\ge  \eta,
\]
as required.

Otherwise, let $d>0$ be arbitrary. Define $x_0:=\sup\{x\in(m-\Delta,m): \, f(x)\le \eta\}$ ($x_0$ is well-defined as the supermum over a non-empty bounded set). Notice that $f>\eta$ on $(x_0,b)$, thus by Lemma~\ref{lem: h lin if not g} there is an affine function $\ell$ such that $h=\ell$ on $(x_0,b)$. By continuity of $\ell$ and $h$, and by the definition of $x_0$, we may choose $a\in\R_+$ so that:
\begin{align}
& m-\Delta<a<x_0, \label{eq: in}\\
&f(a)\le \eta, \label{eq: f<del} \\
& \forall x\in [a,m]: \quad h(m)-h(x)\ge \ell(m)-\ell(x)-\frac {d}{n-1}. \label{eq: h ell}
\end{align}
Let $\bar x=(x_1,\dots,x_{n-1})\in [a,m]^{n-1}$, and write $x_n=\phi_{\bar\lm}(\bar x,m)$. We have:
\begin{align*}
\sum_{j=1}^{n-1} (g(m)-g(x_j)) &\ge \sum_{j=1}^{n-1} (h(m)-h(x_j))
& \textcolor{gray}{\text{maximality of } f(m)} \\
& \ge \sum_{j=1}^{n-1} (\ell(m)-\ell(x_j))-d &\textcolor{gray}{\text{by \eqref{eq: h ell} }} \\
& \ge \ell(x_n)-\ell(m)- d &\textcolor{gray}{\text{by~\eqref{eq: av<m} and \eqref{eq: in}} } \\
& = h(x_n)-h(m)-d    & \textcolor{gray}{x_n\in [m,b] \text{ by \eqref{eq: dual<b} and \eqref{eq: in}, and }h=\ell \text{ on } [m,b]} \\
& \ge g(x_n)-g(m)-d,  & \textcolor{gray}{\text{maximality of } f(m)}
\end{align*}
so $ [a,m]^{n-1}\subseteq L^{g}_{m,d,\bar\lm}$. Also,
\[
g(a)\le h(a)+\eta\le h(m)+\eta \le g(m)-(n-1)\eta \le g(m)-\eta,
\]
so that $|[a,m]|_g\ge \eta$. We conclude that
$ |L^{g}_{m,d,\bar\lm}|_{g^{\PLH (n-1)}} \ge |[a,m]|_{g}^{n-1} \ge \eta^{n-1}\ge \eta$,
as required. Since $d>0$ was arbitrarily chosen, we conclude that
$ |L^{g}_{m,0,\bar\lm}|_{g^{\PLH (n-1)}} \ge\eta$.
The proposition follows.
\end{proof}


\begin{thebibliography}{9}




\bibitem{Aconj}
\textsc{Alon, N.} (2016). \textit{Private communication}.


\bibitem{AY}
\textsc{Alon, N.} and \textsc{Yuster, R.} (1995).
\textit{The 123 Theorem and its extensions},
{J. Comb. Th. Ser. A} \textbf{72}, {322--331}.


\bibitem{social}
\textsc{Alon, N.}, 
\textsc{Feldman, M.}, 
\textsc{Mansour, Y.},
\textsc{Oren, S.} and
\textsc{Tennenholtz, M.} (2016).
\textit{Dynamics of Evolving Social Groups},
{Proc. EC (ACM conference on Economics and Computation)},
{637--654}.

\bibitem{DLL}
\textsc{Dong, Z.},
\textsc{Li, J.},
and \textsc{Li, W.} (2015).
\textit{A note on distribution-free symmetrization inequalities},
{J. Theor. Prob.}
\textbf{28 (3)}, {958--967}.



\bibitem{FFsocial}
\textsc{Feldheim, N.}
and \textsc{Feldheim, O.} (2018).
\textit{Convergence of the quantile admission process with veto power},
preprint arXiv:1807.05550.



\bibitem{K}
\textsc{Katona, G.O.H.} (1985).
\textit{Probabilistic inequalities from extremal graph results (a survey)},
Annals of Discrete Mathematics \textbf{28}, {159--170}.


\bibitem{LM}
\textsc{Li, J.} and
\textsc{Madiman, M.} (2019),
\textit{A combinatorial approach to small ball inequalities for sums and differences},
	Probability and Computing, \textbf{28(1)}, 100--129. 

\bibitem{SW}
\textsc{Siegmund-Schultze, R.} and
\textsc{von Weizs\"{a}cker, H.} (2007).
\textit{Level crossing probabilities I: One-dimensional random walks and symmetrization},
{Adv. Math.} \textbf{208}, 672--679.
\end{thebibliography}
\end{document}